\documentclass[review]{elsarticle}
\usepackage[latin1]{inputenc}
\usepackage[english]{babel}
\usepackage[T1]{fontenc}
\usepackage{amssymb}
\usepackage{amsthm}    % Pretty theorems
\usepackage{stmaryrd} 
\usepackage{amsmath,amsfonts}
\usepackage{graphicx}
\usepackage{epstopdf}
\DeclareGraphicsExtensions{.pdf,.eps,.png,.jpg,.mps}
\usepackage{indentfirst}
\usepackage{dsfont}
\usepackage{color}
\usepackage{enumitem}
\usepackage{float}
\usepackage{pict2e}
\usepackage{soul}
\usepackage{url}

\newtheorem{theorem}{Theorem}[section]
\newtheorem{lemma}[theorem]{Lemma}

\newtheorem{corollary}[theorem]{Corollary}

\newcommand{\paul}[1]{\textcolor{black}{#1}}
\theoremstyle{definition}
\newtheorem{remark}{Remark}[section]

\newtheorem{example}{Example}[section]
\renewcommand{\H}{\mathcal{H}}
\renewcommand{\S}{\mathcal{S}}

\newcommand{\perm}{\mathrm{perm}}

\newcommand{\id}{\mathsf{I}}

\newcommand{\tr}{\mathrm{Tr}}

\usepackage{dsfont}
\newcommand{\abcExample}{{\,\atop \mathord{\includegraphics[height=10ex]{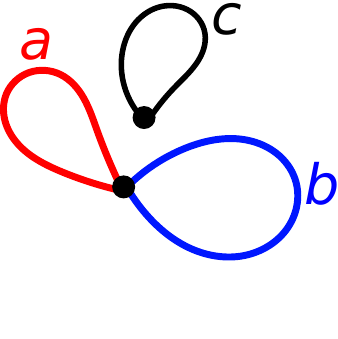}}}}

\usepackage{lineno,hyperref}
\modulolinenumbers[5]

\journal{}%Discrete Mathematics

\begin{document}
\begin{frontmatter}

\title{An Hopf algebra for counting simple cycles}

\author[York]{Pierre-Louis Giscard\corref{Corres}}%\fnref{myfootnote}}
\cortext[Corres]{Corresponding author}
\ead{pierre-louis.giscard@york.ac.uk}
\author[Nantes]{Paul Rochet}
\author[York]{Richard C. Wilson}
\address[York]{University of York, Department of Computer Sciences}
\address[Nantes]{Universit\'e de Nantes, laboratoire de math\'ematiques Jean Leray}

\begin{abstract}
Simple cycles, also known as self-avoiding polygons, are cycles on graphs which are not allowed to visit any vertex more than once. 
We present an exact formula for enumerating the simple cycles of any length on any directed graph involving a sum over its induced subgraphs. This result stems from an Hopf algebra, which we construct explicitly, and which provides further means of counting simple cycles. Finally, we obtain a more general theorem asserting that any Lie idempotent can be used to enumerate simple cycles.  
\end{abstract}


% 05C30: Enumeration in graph theory
% 05C38: Paths and cycles
% 05C22: Signed and weighted graphs
% 05E99: Algebraic combinatorics
% 0611A: Algebraic aspects of posets
% 16T05: Hopf algebra and their applications
\end{frontmatter}

%\linenumbers

\section{Introduction}
Counting simple cycles, that is cycles on graphs which do not visit any vertex more than once, is a problem of fundamental importance with numerous applications in many branches of mathematics. In view of the existing research, this problem should be divided into two main subquestions. One concerns the enumeration of "short" simple cycles with direct applications in the analysis of real-world networks. The other, more related to enumerative combinatorics (see e.g. \cite{Flajolet2009,madras2013self}), concerns the asymptotic growth of the number of simple cycles of length $\ell$ on regular lattices as $\ell$ tends to infinity.

As these two problems have been recognized for a long time, the strategies implemented so far to solve them have been qualitatively different. The practical enumeration of short primes has been tackled via diverse algorithmic and analytic methods, e.g. the inclusion-exclusion principle \cite{bax1993inclusion,Bjorklund2010,karp1982dynamic}, recursive expressions of the adjacency matrix \cite{Khomenko1972}, sieves \cite{bax1996finite} or immanantal equations \cite{cash2007number}. In contrast, the asymptotic growth in the number of long simple cycles on regular lattices has been mostly studied using probability theory \cite{Lawler2004,madras2013self,Copin2012}. 

In our view both problems can be treated with the same tools rooted in the algebraic combinatorics of paths \cite{GiscardRochet2016}. The literature of graph theory contains many different approaches to defining paths on graphs as algebraic objects. A particularly promising one comes from the partially commutative monoid formalism introduced in the 1960s by Cartier and Foata \cite{cartier1969}. Within this framework a path is seen as a word whose letters are the oriented edges of the graph. In the present study we use this formalism to obtain an exact formula for enumerating the simple cycles of any length via a sum over the induced subgraphs of a graph. We then show that this formula stems from an Hopf algebra, which provides further means of counting such cycles. Finally, these results are themselves subsumed under a more general theorem asserting that any Lie idempotent can be used to enumerate the simple cycles.    

The remainder of this article is organized as follows. In section \ref{background} we recall the necessary background concerning Cartier and Foata's formalism as well as recent extensions of it. We proceed in section \ref{PrimeCount} by proving an exact formula for counting the simple cycles relying on this framework. We then show in section \ref{Hopf} that this formula ultimately stems from an Hopf algebra and conclude with a further generalisation to Lie idempotents in section \ref{Lie}.
%
%\pl{A r\'egler encore: $\Lambda(h) = \Lambda(s)$ si on voit $\Lambda$ comme de $\H$ dans $\H$. V\'erifier notation dans tout l'article sur ce point}

%The literature of graph theory contains many different approaches to defining paths on graphs as algebraic objects. A particularly promising one comes from the partially commutative monoid formalism in which a path is seen as a word whose letters are the oriented edges of the graph. 

%\pl{Papers such as this one are difficult to delve into and often require much effort. To alleviate this somewhat, I propose that we show to the reader from the start the kind of results that he will get out of it. The point is to promise a concrete reward whose value the reader can decide to be proportional to the effort it will take him to read through and understand the paper. That is, readers can decide up front whether or not the paper is worth a read for them. 
%In addition, some readers may only want to see the end result, e.g. a formula to get the primes, without having to understand it. I propose that towards the end of the introduction, and after a minimal amount of notation is provided, we give PXperm (one or two versions of it since there are 5 in total) and perhaps mention the Lie idempotent theorem (that any Lie idempotent provides a formula for the primes) since it connects nicely to powerful results in Lie theory.}

\section{Background}\label{background}

Hikes were introduced in a seminal work by Cartier and Foata\footnote{In their original paper in french, Cartier and Foata use the term circuit which was later changed to hike to avoid confusion with other objects in graph theory.} as a generalisation of cycles to possibly disconnected objects \cite{cartier1969}. Relaxing the connectedness condition, Cartier and Foata showed that hikes admit a simple description as words on the alphabet of graph edges and provide a powerful partially commutative framework for algebraic combinatorics on graphs. Recent developments \cite{espinasse2016relations, GiscardRochet2016} have shown that hikes are also words on the alphabet of simple cycles of the graph. In this section we recall a few essential results pertaining to hikes.\\

Let $G=(V,E)$ be a digraph with finite vertex set $V = \{ v_1,...,v_N \}$ and edge set $E$, which may contain loops. Let $\mathsf{W} = (\omega_{ij})_{i,j\,=\,1,\hdots,N}$ represent the weighted adjacency matrix of the graph, built by attributing a formal variable $\omega_{ij}$ to every pair $(v_i,v_j ) \in V^2$ and setting $\omega_{ij}=0$ whenever there is no edge from $v_i$ to $v_j$. In this setting, an edge is identified with a non-zero variable $\omega_{ij}$. A walk of length $\ell$ from $v_i$ to $v_j$ on $G$ is a sequence $p = \omega_{i i_1} \omega_{i_1 i_2} \cdots \omega_{i_{\ell-1} j}$ of $\ell$ contiguous edges. The walk $p$ is \textit{open} if $i \neq j$ and \textit{closed} (a cycle) otherwise. 
A walk $p$ is \textit{simple} if it does not cross the same vertex twice, that is, if the indices $i,i_1,\hdots,i_{\ell-1},j$ are mutually different, with the possible exception $i=j$ if $p$ is closed. Self-loops $\omega_{ii}$ and backtracks $\omega_{ij} \omega_{ji}$ are simple cycles of lengths $1$ and $2$ respectively.

\subsection{Hikes} 
We briefly recall the definition and main properties of hikes. We refer to  \cite{GiscardRochet2016} for further details. Let $\mathcal P$ denote the set of simple cycles in $G$. Hikes are defined as the partially commutative monoid $\mathcal H$ with alphabet $\mathcal P$ and independence relation $ \{ (c, c') \in\mathcal P^2: V(c) \cap V(c') = \emptyset \}$. In less technical terms, a hike is a finite sequence of simple cycles up to permutations of successive vertex-disjoint cycles. 

Hikes form a partially commutative monoid, or trace monoid, when endowed with the concatenation as multiplication and identity element $1$ (the empty hike). Throughout the paper, the concatenation of two hikes $h,h'$ shall be denoted by $h.h'$ or simply $hh'$. If a hike $h$ can be written as the concatenation $h = dd'$ for $d,d' \in \mathcal H$, then we say that $d$ (resp. $d'$) is a left-divisor (resp. right-divisor) of $h$. 
% and denote it by $d |_l h$ (resp. $d' |_r h$). 
Unless stated otherwise, a divisor $d$ of $h$ always refers to a left-divisor and is denoted by $d | h$. 

The prime elements in $(\mathcal H, .)$ are the simple cycles as they verify the prime property: $p | h h'  \Longrightarrow p | h$ or $p | h'$. For this reason, the simple cycles composing a hike $h$ are called the prime factors of $h$. We emphasize that a prime factor may not be a divisor. 

A hike is \textit{self-avoiding} if it does not visit the same vertex twice, i.e. if it is composed of vertex-disjoint simple cycles. Equivalently, $h$ is self-avoiding if $|V(h)| = \ell(h)$ where $V(h)$ denotes the support of $h$ (the set of the vertices it crosses) and $\ell(h)$ its length. By convention, the empty hike is self-avoiding since it has zero length and empty support. 

%Hikes were originally introduced as a generalisation of closed walks (or cycles), that is, finite sequences of simple cycles that can be traveled contiguously on the graph. Thus, while closed walks are indeed special cases of hikes, it is somewhat inconvenient that their characterisation rely on the underlying structure of edges, which are purposely overlooked to define hikes in the most simple way. To overcome this problem, the next result gives a new characterisation of closed walks without having to resort to edges.
%
%
%\begin{proposition}\label{WalkCharacterisation}
%A non-empty hike $h$ is a closed walk if, and only if, it has a unique prime right-divisor.
%\end{proposition}
%
%While the characterisation of closed walks appears to be simple, the proof of this proposition requires some tools from \cite{GiscardRochet2016} that are recalled below. For this reason we postpone it to the end of  the next section.

\subsection{Formal series on hikes}\label{sec:hike_gen}

The hike formalism is perfectly suited to describe the analytic properties of the graph $G$ via its labeled adjacency matrix $\mathsf W$, defined by $\mathsf W_{ij} := \omega_{ij}$ if $(i,j) \in E$ and $\mathsf W_{ij} := 0$ otherwise. In particular, the labeled adjacency matrix $\mathsf{W}$ preserves the partially commutative structure of the hikes \paul{provided that the edges $\omega_{ij}$ are endowed with the commutation rule: $\omega_{ij} \omega_{i'j'} =  \omega_{i'j'}\omega_{ij}$ if $i \neq i'$} \cite{GiscardRochet2016}. Thanks to this property, formal series on hikes can be represented as functions of this matrix and manipulated via its analytical transformations. 
%This way, we remain in a partially commutative framework although the operations on hikes follow from analytical manipulations of $\mathsf W$. 
We recall below some examples from \cite{diekert1990combinatorics,GiscardRochet2016} illustrating these observations.\\

\begin{example}
The trace monoid $\mathcal H$ of hikes forms a partially ordered set, or poset, when the hikes are given an order based on left-divisibility \cite{GiscardRochet2016}. As casually discussed in \cite{cartier1969}, the characteristic function of this poset, i.e. the constant function $\zeta(h)=1$ for $h \in \mathcal H$, is generated from the determinant of the inverse of $(\id - \mathsf W)$, more precisely 
$$ \det \big( \id - \mathsf W \big)^{-1} = \sum_{h \in \mathcal H} \zeta(h) h= \sum_{h \in \mathcal H} h. $$
\paul{In this partially commutative context, $ \det \big( \id - \mathsf W \big)^{-1}$ can be defined formally as the inverse of the hike formal series $\det \big( \id - \mathsf W \big)$, we refer to \cite{GiscardRochet2016} for the technical details.} The characteristic function is also the zeta function of the reduced incidence algebra of $\mathcal H$ as per the now standard terminology introduced by G. C. Rota \cite{rota1987foundations}, which explains the notation. This formula alone highlights the importance of the hike poset $\H$ and its ability to encapsulate the information on the graph structure via its label\paul{ed} adjacency matrix $\mathsf{W}$.
% This formula alone encapsulates and exemplifies the relation between the structure inherent to $\H$ and the digraph sustaining the hikes, represented here by its labels adjacency matrix $\mathsf{W}$.\\
\end{example}
%An analogy with number theory is discussed in \cite{GiscardRochet2016} where a formal variable is introduced so as to recover the series expression of the characteristic function, which is the most common definition of Riemann's zeta function. This is achieved by replacing each edge $\omega_{ij}$ with a formal variable $e^{-s}$, yielding a Laplace-like transformation
%$$ \zeta (s) =  \det \big( \id - e^{-s} \mathsf A \big)^{-1} =  \sum_{h \in \mathcal H} e^{ -s\ell(h)} \ , \  s \in \mathbb R.$$
%Riemann's zeta function is recovered for a graph composed of an (infinite) union of vertex-disjoint cycles \cite{GiscardRochet2016}, a situation in which all hikes commute. In this representation, a simple cycle is associated with a prime number while the length of a hike $h$ plays the role of the logarithm, so as to recover the analogy between $\e^{-s \ell(h)} = (\e^{\ell(h)})^{-s}$ and $n^{-s}$. From this point, numerous common algebraic functions of number theory have their hike counterparts, giving rise to interesting interpretations. 
%
%\begin{remark} Because hikes cannot be identified by their length, the Laplace-like transformation $f \mapsto \sum_{h\in \mathcal H} f(c) \e^{- \ . \ \ell(h)}$ is not one-to-one. For the zeta function in particular, this means that considering $\e^{-s} \mathsf A$ instead of $\mathsf W$ induces a loss of information \cite{GiscardRochet2016}. Nevertheless, 
%this transformation allows to investigate the hikes by length, which is sufficient for our purposes.
%\end{remark}

\begin{example}\label{MobiusEx}
%The zeta function introduced above was originally obtained from the inverse of the M\"{o}bius function. 
A standard result in trace monoid theory states that the M\"{o}bius function\paul{, that is the inverse of the zeta function,} of such a monoid is expressed as a series over words composed of different commuting letters, see for instance \cite{diekert1990combinatorics}. For the trace monoid $\H$, where \paul{different} hikes commute if they are vertex-disjoint, the M\"{o}bius function is given by
$$ \mu(h) = \left\{ \begin{array}{cl} (-1)^{\Omega(h)} & \text{if $h$ is self-avoiding} \\ 0 & \text{otherwise}  \end{array} \right. \ , \ h \in \mathcal H  $$
where $\Omega(h)$ counts, with multiplicity, the number of prime factors of $h \in \mathcal H$. The M\"{o}bius function admits an expression involving the labeled adjacency matrix $\mathsf W$, namely
	$$ \det(\id - \mathsf W) = \sum_{\substack{h \in \mathcal H}} \mu(h) h. $$
%	From now on, we will abuse the notation slightly and write $\mu=\det(\id - \mathsf W)$.
This relation has been discussed under many forms in the literature, see e.g. \cite{cartier1969,choffrut1999determinants,diekert1988transitive,GiscardRochet2016}. The immediate consequence on the divisors of a hike arises from writing
\begin{equation}\label{inv_mob} 1 =  \det(\id - \mathsf W) .\det(\id - \mathsf W)^{-1} = \sum_{\substack{h \in \mathcal H}} \mu(h) h . \sum_{h \in \mathcal H} h  = \sum_{h \in \mathcal H} \bigg( \sum_{d | h} \mu(d) \bigg) h. \end{equation}
From this we deduce the M\"{o}bius inversion formula $\sum_{d | h} \mu(d) = 0$ whenever $h \neq 1$.\\
\end{example}

\begin{example} The hike von Mangoldt function arises from the trace of $ (\id - \mathsf W)^{-1}  - \id$, 
$$ \tr \big( (\id - \mathsf W)^{-1}  - \id \big) = \sum_{h \in \mathcal H} \Lambda(h) h.  $$
The diagonal of $(\id - \mathsf W)^{-1} - \id=  \mathsf W + \mathsf W^2 + ...$ only involves non-empty walks, for which a contiguous representation can be put in one-to-one correspondence with a starting vertex. Thus, $\Lambda(h)$ counts the number of contiguous representations of $h$, that is the number of ways to write $h$ as a succession of contiguous edges (in particular $\Lambda(h) = 0$ if $h$ is not a walk). Remark that in a graph where all simple cycles are disjoint, the number of contiguous representations of a closed hike is $\Lambda(h) = \ell(p)$ if $h = p^k$ for $p$ a simple cycle and $k \geq 1$ and $\Lambda(h)=0$ otherwise. This highlights the relation with number theory discussed in \cite{GiscardRochet2016} where the von Mangoldt function, obtained from a log-derivative of the zeta function, draws a parallel between the length of a hike and the logarithm of an integer. An important consequence is the following expression of the length as a product of the von Mangoldt function and the zeta function, 
\begin{equation} \label{van_m} 
 \sum_{h \in \mathcal H} \Lambda(h) h .  \sum_{h \in \mathcal H} h = \sum_{h \in \mathcal H} \bigg( \sum_{d | h} \Lambda(d)  \bigg) h = \sum_{h \in \mathcal H} \ell(h) h,
\end{equation}
which recovers after identification $\sum_{d | h} \Lambda(d) = \ell(h) \ , \ h \in \mathcal H$. 
\end{example}

\begin{example}\label{LambdaEx} The hike Liouville function is defined by
$$ \lambda(h) = (-1)^{\Omega(h)} \ , \ h \in \mathcal H.  $$
As in the number-theoretic version, the inverse of the Liouville function is the absolute value of the M\"{o}bius function (see Proposition 3.10 in \cite{GiscardRochet2016}). In the graph context, the absolute value of the M\"{o}bius function attributes the value $1$ to every self-avoiding closed hike and thus writes as the permanent of $\id + \mathsf W$:
$$ \perm( \id + \mathsf W) = \sum_{\substack{h \in \mathcal S}} h = \sum_{\substack{h \in \mathcal H}} | \mu(h) | h = \frac{1}{\sum_{h \in \mathcal H} \lambda(h)  h}, $$
where $\mathcal S$ denotes the set of self-avoiding hikes. It is somewhat remarkable that the inverse relation between $\lambda$ and $| \mu |$, which holds in the fully commutative poset of the integers, is still verified in this more general partially commutative framework. 
\end{example}
%~
%
%\noindent We are now in position to prove Proposition \ref{WalkCharacterisation}.\\
%
%\begin{proof}[Proof of Proposition \ref{WalkCharacterisation}] 
%A hike $h$ is a non-empty closed walk if and only if it has a contiguous representation, i.e. if its image $\Lambda(h)$ through the hike von Mangoldt function is non-zero. From Proposition 3.4 in \cite{GiscardRochet2016}, we know that
%$$ \Lambda(h) = \sum_{d | h} \ell(d) \mu \Big( \frac h d \Big),  $$ 
%where $\mu$ is the hike M\"{o}bius function  (this result simply follows from the right-multiplication by $\sum_{h \in \mathcal H} \mu(h) h$ in Eq. \eqref{van_m}).  Assume $h \neq 1$ and let $p_1,\cdots,p_k$ be the different prime right-divisors of $h$. Remark that the $p_j$'s are vertex-disjoint so that $s := p_1\cdots p_k$ is self-avoiding and the self-avoiding right-divisors of $h$ are the divisors of $s$. Write $h = h' s $, since $\mu(h/d)$ is zero if $h/d$ is not self-avoiding, we deduce
%$$ \Lambda(h) = \sum_{d | h' s} \ell(d) \mu \Big( \frac {h's} d \Big) = \sum_{d' | s} \ell(h' d') \mu \Big( \frac {s} {d'} \Big).  $$ 
%By additivity of $\ell(.)$, we get
%$$ \Lambda(h) = \ell(h') \sum_{d' | s} \mu \Big( \frac {s} {d'} \Big) + \sum_{d' | s} \ell(d') \mu \Big( \frac {s} {d'} \Big) =  \Lambda(s),  $$ 
%where the first term vanishes due to the M\"{o}bius inversion formula. It remains to notice that $\Lambda(s) = \ell(s)$ if $s$ is prime and $\Lambda(s)=0$ if $s$ has more than one prime divisor.
%\end{proof} 

\section{Counting primes via a convolution over induced subgraphs}\label{PrimeCount}

Let $\mathcal G$ be the set of finite digraphs. For $G = (V(G),E(G)) \in \mathcal G$, we say that $H = (V(H), E(H)) \in \mathcal G$ is an induced subgraph of $G$ (which we denote by $H \prec G$) if $V(H) \subseteq V(G)$ and $E(H) = E(G) \cap V(H)^2$. If $H \prec G$, then $G-H$ designates the subgraph of $G$ induced by $V(G) \setminus V(H)$. 
Let $(\mathcal A,.,+)$ be an algebra, for two functions $\phi, \psi: \mathcal G \to \mathcal A$, the induced subgraph convolution between $\phi$ and $\psi$ is defined by
$$ (\phi * \psi) [G] = \sum_{H \prec G} \phi[H] \psi[G-H] \ , \ G \in \mathcal G,  $$
where the sum runs over all induced subgraphs of $G$ including the empty graph $\emptyset$ and $G$ itself. In this section, we investigate the induced subgraph convolution between functions with values in the algebra $\mathbb R\langle \mathcal H \rangle$ of formal series on hikes with real coefficients. Examples of such functions arising from usual expressions of the labeled adjacency matrix $\mathsf W_H$ of a digraph $H$ have been discussed in Section \ref{sec:hike_gen}. 

\paul{\begin{remark}\label{ConvoRemark} For all functions $\phi: \mathcal G \to \mathbb R\langle \mathcal H \rangle$ considered in this paper, the formal series $\phi[H]$ only involves hikes $h$ whose support $V(h)$ lies in $H$. In other words, the coefficient of $h$ in $\phi[H]$ is zero whenever $V(h) \nsubseteq V(H)$. Consequently, the convolution $\phi * \psi$ between two such functions $\phi,\psi$ is always commutative.
\end{remark}}

\begin{lemma}\label{lem:1} For all $G \in \mathcal G$, 
$$\sum_{H \prec G} \det (- \mathsf W_{H} )\, \perm (\mathsf W_{G-H} ) = \delta[G] := \left\{ \begin{array}{cl} 1 & \text{if } G = \emptyset \\ 0 & \text{otherwise,}\end{array} \right. $$
where we use the convention $\perm (\mathsf W_\emptyset )= \det (- \mathsf W_{\emptyset} )=1$.
\end{lemma}

The function $\delta$ is the identity function for the induced subgraph convolution $*$, in view of $\phi*\delta = \phi$ for all $\phi: \mathcal G \to \mathbb R\langle \mathcal H \rangle$. Thus, the lemma establishes that the functions $G \mapsto \perm (\mathsf W_G )$ and  $ G \mapsto \det (- \mathsf W_{G} )$ are mutual inverse through $*$. \\

\begin{proof}[Proof of Lemma \ref{lem:1}]
Recall that $\mathcal S $ is the set of self-avoiding hikes on $G$ and let $\mathcal S_H$ denote the set of self-avoiding hikes with support $V(H)$, for $H \prec G$.  Both permanent and determinant have simple expressions in terms of self-avoiding hikes. 
$$ \det (- \mathsf W_{H}) = \sum_{h \in \mathcal S_H}(-1)^{\Omega(h)} h \ \ \text{ and } \ \ \perm (\mathsf W_H) =  \sum_{h \in \mathcal S_H }h, $$
where we recall that $\Omega(h)$ is the number of simple cycles composing $s$. For a self-avoiding hike $s$, each divisor $d$ of $s$ can be put in one-to-one correspondance with the subgraph induced by its support. Thus,
\begin{align*} \sum_{H \prec G} \det ( - \mathsf W_{H})\perm (\mathsf W_{G-H})  & = \sum_{H \prec G} \bigg(  \sum_{h\in \mathcal S_H} (-1)^{\Omega(h)} h . \sum_{h\in \mathcal S_{G-H}}  h \bigg) \\
& = \sum_{h \in \mathcal S_{G}} \bigg( \sum_{d | h} (-1)^{\Omega(d)} \bigg) h.
\end{align*}
It remains to notice that for a self-avoiding hike $h$, $\sum_{d | h} (-1)^{\Omega(d)} = \sum_{d | h} \mu(d) = 1$ if $h=1$ and $0$ otherwise, by Eq. \eqref{inv_mob}.
\end{proof}
~

\begin{corollary}\label{cor:1} For all $G \in \mathcal G$, 
$$\sum_{H \prec G} \perm (\mathsf W_H ) \det (\id - \mathsf W_{G-H} ) = \sum_{H \prec G} \perm (\id + \mathsf W_H ) \det (- \mathsf W_{G-H} ) = 1.$$
\end{corollary}

\noindent Seeing these sums as convolutions makes the proof almost trivial. 

\begin{proof} First observe that, because $\mathcal S $ is the disjoint union of the $\mathcal S_H$ for $H \prec G$, one has
$$ \sum_{H \prec G} \perm( \mathsf W_H) = \sum_{H \prec G} \sum_{h \in \mathcal S_H} h = \sum_{h \in \mathcal S} h = \perm(\id + \mathsf W_G) . $$ 
A similar observation can be made for $\det(\id - \mathsf W_G)$,
$$ \sum_{H \prec G} \det( -\mathsf W_H) = \sum_{H \prec G} \sum_{h \in \mathcal S_H} (-1)^{\Omega(h)} h = \sum_{h \in \mathcal S} (-1)^{\Omega(h)} h = \det(\id - \mathsf W_G). $$
So, letting $ \phi[H] = \det(-\mathsf W_H)$, $\psi[H] =\perm(\mathsf W_H)$ \paul{and $\underline 1[H] = 1$}, the equations of Corollary \ref{cor:1} read 
$$ \psi *(\phi* \underline 1) = \underline 1 \ \text{ and } \ \phi * (\psi* \underline 1) = \underline 1.   $$
The result follows directly from Lemma \ref{lem:1}, using the \paul{associativity} and commutativity of the convolution, e.g. 
$$ \psi *(\phi*\underline 1) = (\psi * \underline 1) * \phi = (\psi * \phi) * \underline 1 = \delta * \underline 1 = \underline 1. $$
\end{proof}

\noindent We now derive an expression of the formal series of Hamiltonian cycles. In the spirit of \cite{menous2013logarithmic}, we introduce the derivation operator $D$ defined by
$$ 
D\sum_{h \in \mathcal H} f(h) h =  \sum_{h \in \mathcal H} \ell(h) f(h) h . 
$$
~

\begin{theorem}\label{Hamiltonian} 
Let $\mathcal P_G$ denote the set of primes with support $V(G)$, that is the set of
Hamiltonian cycles on $G$. Then
\begin{align*} D \sum_{p \in \mathcal P_G} p  & = \sum_{H \prec G} \det( - \mathsf W_{H}) \,D \, \perm( \mathsf W_{G-H}), \\
& = - \sum_{H \prec G}  \perm(  \mathsf W_{H}) \,D\det( \mathsf W_{G-H}) .
\end{align*}
\end{theorem}

\begin{proof} Let $\pi[G] = \sum_{p \in \mathcal P_G}p $, the first equality of the theorem writes
$$ D \pi = \phi *D \psi.  $$
Calculating explicitly $\phi *D \psi$ gives
\begin{align*} \sum_{H \prec G} \det(-\mathsf W_H) \,D \, \perm(\mathsf W_{G-H})  & = \sum_{H \prec G} \bigg( \sum_{h \in \mathcal S_H} (-1)^{\Omega(h)}h . \sum_{h \in \mathcal S_{G-H}} \ell(h) h \bigg), \\
& = \sum_{h \in \mathcal S_G} h \sum_{d | h} (-1)^{\Omega(d)} \ell \Big( \frac h d \Big) .
\end{align*}
The result follows from noting that $\sum_{d | h} (-1)^{\Omega(d)} \ell ( h/d)$ is none other than the hike von Mangoldt function $\Lambda(s)$ of Eq. \eqref{van_m}. Since $h$ is self-avoiding, $\Lambda(h)$ is equal to $\ell(h)$ if $h$ is connected and $0$ otherwise. Thus,
$$ (\phi *D   \psi )[G] = \sum_{H \prec G} \det(-\mathsf W_H)\, D \,  \perm(\mathsf W_{G-H})  = \sum_{p \in \mathcal P_G } \ell(p)  p = D \,\pi[G].$$
For the second equality, simply observe that $ - \psi*D \phi =  \phi*D  \psi  - D(\phi*\psi) = \phi*D \psi -D \underline 1 =\phi*D \psi $.
\end{proof}
~

\begin{remark} Because $\phi$ is the inverse of $\psi$, the relation $D \pi=\phi*D \psi$ suggests an expression of $\pi$ as a logarithm of $\psi$. This is indeed the case. Observe that the $k$-times convolution
$$ 
\pi^{*k}[G] := \underbrace{\pi* \cdots * \pi}_{k \text{ times}} \ [G] = \sum_{(H_1,\,\hdots,\,H_k)} \pi[H_1]\cdots\pi[H_k],  
$$
writes as the sum over all $k$-partitions $H_1,...,H_k$ of $G$ (here, the order is important meaning that there are $k!$ partitions involving the subgraphs $H_1,\,\hdots,\,H_k$). Thus, every spanning self-avoiding hike $h$ is counted exactly once in the exponentiation
$$ \exp_*(\pi[G]) = \sum_{k \geq 0} \frac 1 {k!} \pi^{*k} [G] = \psi[G].  $$
This aspect originates from an Hopf algebraic structure, which we describe in the next section.
\end{remark}

The formal series of simple cycles (of any length) follows from the convolution of $\pi$ with the \paul{constant function $\underline 1$},
$$ \Pi[G] := (\pi* \underline 1 )[G] = \sum_{H \prec G} \pi[H]  = \sum_{p \in \mathcal P} p, $$
Because derivation and convolution with the constant are commuting operators, we recover
$$ D\,\Pi = D (\pi* \underline 1 ) = D \pi* \underline 1 = (\phi * D\psi) * \underline  1 = \phi* D (\psi* \underline 1 ) = \phi* D \Psi, $$ 
where $\Psi[G] = (\psi* \underline 1)[G] = \sum_{H \prec G} \perm( \mathsf A_H) = \perm(\id + \mathsf A_G)$. 
This gives the following corollary to Theorem \ref{Hamiltonian}
\begin{corollary} Let $\mathcal{P}$ be the set of all primes on $G$, then
$$ D \sum_{p \in \mathcal P} p = \sum_{H \prec G} \det( -  \mathsf W_{H}) \,D\, \perm(\id+ \mathsf W_{G-H}).
$$
\end{corollary}
~

\section{An Hopf algebra structure for the hikes}\label{Hopf}
%\pl{NOTATION I used that should be defined earlier: $\mathcal{P}$ for the sub-monoid of the primes, $\mathcal{C}$ for the sub-monoid of connected hikes. I also needed $\det(I-zW)$ is M\"{o}bius, 1/det the zeta function and $\perm(I-zW)$ the identity on $\mathcal S$ and finally, $\ell(h)$ as the length of $h$. This would be best explained/ defined before the present section which I am struggling to maintain short. I also talk about the "algebra of functions on hikes endowed with induced subgraph convolution", I am not sure this is defined earlier.}

Since any hike on a graph can be seen as a disjoint ensemble of connected components, it is natural that an algebraic structure should exist describing the generation of arbitrary hikes from connected ones. In particular, when it comes to the self-avoiding hikes, their connected components are their prime factors. Therefore, if this algebraic structure provides a mean of projecting the set of hikes back onto the set of connected hikes, it might send the self-avoiding ones onto the primes. In this section, we establish these heuristic arguments rigorously by showing that a commutative and cocommutative Hopf algebra describes the generation of self-avoiding hikes from simple cycles. This algebra provides several exact formulas for the formal series of primes, stemming from projectors onto the irreducible elements of the algebra. We also show that the subgraph convolution operation introduced previously is a necessary and unavoidable feature resulting from this algebra.\\  

%\pl{We denote $K$ a field and $K\langle \H\rangle$ the set of all set-theoretical maps $\H\to K$. For any $f\in K\langle \H\rangle$, we represent $\sum_{h\in \H} f(h)\,h$.
%$K\langle \H\rangle$ is made into an algebra by  linearly and continuously lifting the multiplication from $\H$}
 
In this section, we consider finite graphs. Observe in particular that on such a graph $G$ the set $\S_G$ of self-avoiding hikes has finite cardinality. 
% Let $K$ be a field,
%an algebraically closed field of characteristic zero, 
%and let $K\langle \S \rangle$ denote the monoid algebra of $\S$ over $K$, that is the set of %formal $K$-linear sums over $\S$. For any set-theoretical map $f:~\S\mapsto K$, we denote by %$\sum_{h\in \S} f(h)\, h \in K\langle \S\rangle $ the associated $K$-linear combination.  
Since $\S_G$ is not closed under the ordinary multiplication between hikes, it is necessary to introduce another multiplication between elements of $\S$ under which it is closed, so as to obtain an algebra structure. Thus, we define:
$$h,h'\in \S_G\Rightarrow h\ast h':=\begin{cases}hh',&\text{if $V(h)\cap V(h')=\emptyset$}\\0,&\text{otherwise}.\end{cases}$$
With the convention that $0\in \S_G$ and since $hh'=h'h\iff V(h)\cap V(h')=\emptyset$, it is clear that $(\S_G,\ast)$ forms a commutative algebra.
Although seemingly arbitrary at first sight, we will see that the $\ast$ multiplication induces a natural convolution between elements of the monoid algebra $K\langle \S_G\rangle$ of $\S_G$ over a commutative ring $K$ with a unit,\footnote{In other terms $K\langle \S_G\rangle$ is simply the set of finite $K$-linear combinations of elements of $\S_G$.} which coincides with the \emph{induced subgraph convolution} introduced in Section~\ref{PrimeCount}, thanks to which $\ast$ will be easy to implement in practice.\\ 

To endow $\S_G$ with an Hopf algebra structure, we follow Schmitt's construction for general trace monoids \cite{Schmitt1990}, and begin by defining a comultiplication $\Delta: \S_G\to \S_G\otimes \S_G$ and a counit $\epsilon:  \S_G\to K$ as follows
\begin{align}\label{DeltaE}
\epsilon(h) = \begin{cases}1,&\text{if $h=1$,}\\0,&\text{otherwise,}\end{cases}\qquad \text{and}\qquad \Delta (h) = \sum_{\substack{d|h \\V(d)\,\cap\, V(h/d)=\emptyset}} d\otimes \frac{h}{d}.
\end{align}
The comultiplication introduced above decomposes self-avoiding hikes into their divisors. It can be extended to the set $\H_G$ of all hikes on $G$, in which case it decomposes hikes into disjoint-divisors. 
%The condition $V(d)\,\cap\, V(h/d)=\emptyset$ reflects the opposite requirement, namely that $d$ and $h/d$ comprise disjoint connected components of $h$ and hence, are themselves disjoint. The comultiplication introduced above thus decomposes $h$ into disjoint divisors. 

\begin{remark}
The comultiplication introduced above recovers that defined by Schmitt on general trace monoids in \cite{Schmitt1990}. Let $h=c_1\cdots c_n$ be a self-avoiding hike with disjoint connected components $c_1,\cdots, c_n$, and let $U:=\{i_1,\cdots, i_k\}$ be a subsequence of $N:=\{1,\cdots, n\}$. The complement of $U$ in $N$ is denoted $\bar{U}$. For any $h\in \S$, let $h|_U:=c_{i_1}\cdots c_{i_k}$, in particular if $U$ is empty we set $h|_\emptyset:=1$. 
Schmitt then defines the comultiplication by
\begin{align*}
\Delta(h) := \sum_{U\subseteq N}h|_U\otimes h|_{\bar{U}}.
\end{align*}
Now observe that since $h$ is self-avoiding, for any divisor $d$ of $h$ we have $V(d)\,\cap\, V(h/d)=\emptyset$. In particular, it must be that $d=c_{i_1}\cdots c_{i_k}=h|_U$ for some $U$, so that Schmitt's definition is seen to be equivalent to Eq.~(\ref{DeltaE}). Thanks to this observation we can directly use Schmitt's results, thereby alleviating a number of proofs.
\end{remark}
~

Equipped with these operations, $ \S_G$ forms a \emph{cocommutative coalgebra} \cite{Schmitt1990}. Its irreducible elements, i.e. those that fulfill $\Delta(h)=1\otimes h+h\otimes 1$, are immediately seen from Eq.~(\ref{DeltaE}) to be the connected self-avoiding hikes, that is the primes, since these have no non-empty disjoint divisors. This further confirms that $\Delta$ pertains to the generation of hikes from connected ones.\\
~

We may now invoke the general results of \cite{Schmitt1990} to observe that the algebraic and coalgebraic structures of $\S_G$ are compatible, that is $\Delta$ and $\epsilon$ are algebra maps and $(\S_G,\ast,\Delta,\epsilon)$ is a bialgebra, which we will simply denote $\S_G$ to alleviate the notation. These results are subsumed in the following theorem, which in addition to the bialgebra structure, provides an antipode for $\S_G$, turning it into a Hopf algebra.
 
\begin{theorem}\label{HopfTheorem}
$\S_G$ is a cocommutative Hopf algebra, with comultiplication and counit defined above and antipode $S$ given by $S(h):=(-1)^{c(h)}h$, where $c(h)$ is the number of disjoint connected components of $h$. 
\end{theorem}
\begin{proof}
As stated earlier, the results of Schmitt ensure that with the definitions of Eq.~(\ref{DeltaE}), $\S_G$ is a cocommutative bialgebra \cite{Schmitt1990}. It remains to show that $S(h)$ given above is indeed an antipode, that is the inverse of the identity on $\S_G$. This follows immediately from Proposition 3.3 of \cite{Schmitt1990} and Example 4.2 of \cite{Schmitt1994}, on noting that the decomposition of any hike into its disjoint connected components is unique \cite{GiscardRochet2016} and that disjoint connected components commute, so that $\tilde{h}:=c_n c_{n-1}\cdots c_1 = h$.
Alternatively, observe that the $S$ coincides with the M\"{o}bius function of Example~\ref{MobiusEx}, that is $S(h)=\mu(h)\,h$, since $c(h)=\Omega(h)$ whenever $h\in \S_G$. It follows that $S$ is the inverse of the identity on $\S_G$.
\end{proof}
~

Since $\S_G$ is a commutative and cocommutative Hopf algebra and $\S_G$ has finite cardinality, it follows by linearity that the monoid algebra $K\langle \S_G\rangle$ is a commutative and cocommutative Hopf algebra too. The convolution between elements of $K\langle\S_G\rangle$ is obtained by continuously and linearly lifting the multiplication introduced earlier between elements of $\S_G$. Equivalently, the convolution can be constructed explicitely from the comultiplication defined earlier on $\S_G$.\\

Let $a,b:\,\S_G\to K$, and define $\alpha, \beta:\mathcal{G}\to K\langle \S_G\rangle$ with $\alpha[G] = \sum_{h\in \S_G}a(h)h$ and $\beta[G] = \sum_{h\in \S_G}b(h)h$. In $K\langle\S_G\rangle$, the multiplication between $\alpha[G]$ and $\beta[G]$ is obtained from the comultiplication as $\alpha[G]\ast\beta[G] = \sum_{h\in \S_G}(a\ast b)(h) \,h$, where $(a\ast b)(h)\,h:=\left(m \circ (a\otimes b)\circ \, \Delta\right)(h)$ and $m$ is the ordinary multiplication between hikes \cite{Schmitt1987, Schmitt1990}. This gives
\begin{align*}
\alpha[G]\ast\beta[G] = \sum_{h\in \S_G}(a\ast b)(h)\,h&\,=\sum_{h\in \S}\,\,\sum_{\substack{d|h \\V(d)\,\cap\, V(h/d)=\emptyset}}a(d)b(h/d)\,h.
\end{align*}
Since $V(d)\,\cap\, V(h/d)=\emptyset$, then $h=d\ast \frac{h}{d}=d\,\frac{h}{d}$ and we can write
\begin{align*}
\alpha[G]\ast \beta[G] &=\sum_{H\prec G}~~\sum_{\substack{d:V(d)\subseteq V(H)}}a(d)\, d~~\sum_{d':V(d')\subseteq V(G-H)}b(d')\,d'.
\end{align*}
Observe that, per Remark~\ref{ConvoRemark}, $\sum_{\substack{d:V(d)\subseteq V(H)}}a(d)\,d$ and $\sum_{d':V(d')\subseteq V(G-H)}b(d')\,d'$ are $\alpha[H]$ and $\beta[G-H]$ respectively. Thus we have obtained the convolution between elements of $K\langle \S_G\rangle$ as 
$$
\alpha[G]\ast \beta[G] = \sum_{H\prec G} \alpha[H]\,\beta[G-H] = (\alpha\ast \beta)[G],
$$
that is $\alpha[G]\ast \beta[G]$ is the induced subgraph convolution of $\alpha $ with $\beta$. This result establishes that the induced subgraph convolution arises directly  from the definition of the comultiplication in Eq.~(\ref{DeltaE}). This, in turn, shows that it is an unavoidable feature reflecting the generation of hikes from their disjoint connected components.

~\\
%To proceed further, consider the one-to-one mapping which to any $f\in K\langle \H\rangle$ associates a function on graphs $\hat{f}:\mathcal{G}\to  K\langle \H\rangle$ such that $\hat{f}[G]:=\sum_{h\in \H_G}f(h)$, $\H_G$ being the trace monoid of hikes with support $V(G)$. \pl{Map $f\mapsto \hat{f}$ bijective?}
%Then Eq.~(\ref{etimesf}) indicates that 
%\begin{align*}
%\widehat{(f\ast g)}[G] = \sum_{h\in \H_G} \Bigg(\sum_{\substack{d|h \\V(d)\,\cap\, V(h/d)=\emptyset}}f(d)g(h/d)\Bigg),
%\end{align*}
%while 
%\begin{align*}
%(\hat{f}\ast \hat{g})[G] &= \sum_{H\prec G}\hat{f}[H]\hat{g}[G-H] = \sum_{H\prec G}\bigg(\sum_{h\in \H_H} f(h) \sum_{h'\in \H_{G-H}}g(h')\bigg).
%\intertext{Now since the product of a hike of $\H_H$ with a hike of $\H_{G-H}$ is a hike of $\H_G$, it follows that }
%(\hat{f}\ast \hat{g})[G] & = \sum_{h\in \H_G} \Bigg(\sum_{\substack{d|h \\V(d)\,\cap\, V(h/d)=\emptyset}}f(d)g(h/d)\Bigg) = \widehat{(f\ast g)}[G].
%\end{align*}
%Thus the bijective mapping $f\mapsto \hat{f}$ is an algebra homomorphism, which concludes the proof. 
%From now on, we shall the induced subgraph convolution and the $\ast$ multiplication between elements of $\S$.

%A consequence of Lemma~\ref{SubgraphConvoIsomorph} is that the identity on $K\langle \S\rangle$ is $\sum_{h\in \S} h$, which is the zeta function

The generation of self-avoiding hikes from connected self-avoiding ones, that is simple cycles, thus gives rise to an Hopf algebra which, we will see, provides means of doing the opposite, that is to obtain the simple cycles from the set of self-avoiding hikes. 
%\begin{corollary}\label{HsaHopf}
%$K\langle \mathcal S\rangle$ is a cocommutative sub-Hopf algebra of $K\langle \H\rangle$.
%\end{corollary}
%\begin{proof}
%$K\langle \mathcal S\rangle$ inherits a bialgebraic structure from $K\langle \H\rangle$, in particular since it is clearly stable under induced subgraph convolution. Thus it suffices to verify that it has an antipode and this follows once more from the general results of Schmitt \cite{Schmitt1990}, with the antipode being that of $K\langle\H\rangle$ restricted to $K\langle \mathcal S\rangle$. Alternatively, we can verify this result directly on noting that for $h\in \mathcal S$, $c(h)=\Omega(h)$, so that the antipode of $K\langle\mathcal S\rangle$ is just the M\"{o}bius function which, by Corollary~\ref{cor:1}, is the $\ast$-inverse of the identity on $\mathcal S$, as required.
%\end{proof}

\begin{theorem}\label{thm:main}
For any coalgebra map $f: \S_G\to \S_G$, then $\log_\ast f: \S_G\to \mathcal{P}_G$. In particular $\log_\ast \text{id}_{\S_G}$, the $\ast$-logarithm of the identity on $\S_G$ is the projector onto $\mathcal{P}_G$.
By linearity, the $\ast$-logarithm also projects $K\langle \S_G\rangle$ onto $K\langle \mathcal{P}_G\rangle$, in particular
$$
\log_{\ast}\sum_{h\in \S_G}h=\sum_{p\in \mathcal{P}_G}p.
$$
Equivalently,
$$
\log_{\ast}\!\big(\perm(\mathsf{I}+\mathsf{W})\big)=-\log_{\ast}\!\big(\det(\mathsf{I}-\mathsf{W})\big)=\Pi[G].
$$
\end{theorem}
\begin{proof}
We rely on the results of Schmitt \cite{Schmitt1994} concerning cocommutative Hopf algebras, more precisely Theorem 9.4, 9.5, Corollary 9.6 and Example 9.2. Accordingly, the image of the $\ast$-logarithm of any coalgebra map $f:\,H\to H$ of a cocommutative Hopf algebra $H$ is in the sub-monoid of irreducible elements of $H$, $\text{Irr}(H)$. In particular the $\ast$-logarithm \emph{of the identity} on $H$ is the identity on $\text{Irr}(H)$.
In the present context, $H=\S_G$ and $\text{Irr}(H)=\mathcal{P}_G$. By linearity this extends to $H=K\langle \S_G\rangle$ and $\text{Irr}(H)=K\langle \mathcal{P}_G\rangle$ as well. 
Now the identity on $K\langle \S_G\rangle$, $\sum_{h\in S}\text{id}_{\S_G}(h)$, is the zeta function of $\S_G$, i.e. $\zeta_{\S_G} =\perm(\mathsf{I}+\mathsf{W})$ per Example~\ref{LambdaEx}. This gives the first result. 
The second equality follows from the observation that $\sum_{h\in \S_G}S(h)=\sum_{h\in \S_G}\mu(h)\,h$ is the antipode of $K\langle \mathcal{S}_G\rangle$, i.e. the $\ast$-inverse of the identity on $K\langle\S_G\rangle$. As shown in Example~\ref{MobiusEx}, this is $\det(\mathsf{I}-\mathsf{W})$. This result was also shown directly in Corollary~\ref{cor:1}.
\end{proof}

\begin{remark}\label{GenRel}
The $\ast$-inverse relation between the permanent and the determinant is a special case of a more general identity, which is easily obtained either directly for each $k\in \mathbb{Z}$ or by induction from the base cases $k=\pm 1$:
$$
\perm(\mathsf{I}+\mathsf{W})^{\ast k} = \sum_{h\in \S_G} k^{\Omega(h)} \, h,\quad k\in \mathbb{Z}.
$$
Setting $k=-1$ recovers $\perm(\mathsf{I}+\mathsf{W})^{\ast-1}=\det(\mathsf{I}-\mathsf{W})$, while $k=0$ gives $\perm(\mathsf{I}+\mathsf{W})^{\ast0}=1$ with the convention that $0^0=1$. 
%The correspondance between algebraic combinatorics on hikes and number theory laid out in \cite{GiscardRochet2016} also indicates that the case $k=2$ 
%
%
%
\end{remark}

%\begin{remark}
%Before we proceed to these exemples and further formulas, remark that these extend naturally to the entire monoid $\H$, with $(\H,\ast,\Delta,\epsilon)$ a cocommutative Hopf algebra relating arbitrary hikes with connected ones. 
%\end{remark}

\begin{example}[Simple cycles from the logarithm of the determinant]
Let us illustrate how the logarithm with respect to induced subgraph convolution distillates the simple cycles from a determinant or a permanent. Consider the following graph with three simple cycles $a$, $b$ and $c$ of arbitrary lengths:\\[-5ex]
\begin{center}$G=\abcExample$\end{center}
\vspace{-4mm}
for which $\det(\mathsf{I}-\mathsf{W}) = 1-a-b-c+ac+bc$. Expanding the logarithm of this determinant as a series and focusing on the first and second orders to begin with, we have 
\begin{equation}\label{LogStep1}
-\log_{\ast}\!\big(\det(\mathsf{I}-\mathsf{W})\big) = -(-a-b-c+ac+bc)+\frac{1}{2}(-a-b-c+ac+bc)^{\ast 2}-\cdots
\end{equation}
Since $V(a)\cap V(c) = V(b)\cap V(c)=\emptyset$ and $V(a)\cap V(b) \neq\emptyset$, terms such as $a\ast a$, $a\ast b$ and $b\ast ac$ all vanish and only $a\ast c$ and $b\ast c $ are non-zero. 
Thus, expanding the second order leaves $ac +ca+ bc+cb$. For the same reasons all higher orders of the logarithm are exactly zero. 
In addition, since $a$ and $c$ and $b$ and $c$ are vertex-disjoint, they commute, and the second order further simplifies to $2ac+2bc$. Thanks to these observations, Eq.~(\ref{LogStep1}) becomes
\begin{align*}
-\log_{\ast}\!\big(\det(\mathsf{I}-\mathsf{W})\big) &= a+b+c-ac-bc+\frac{1}{2}(2ac+2bc),\\
&= a+ b +c,
\end{align*}
which is indeed the formal series of the primes on $G$.
\end{example}
~

Admittedly, a $\ast$-logarithm is not very convenient to implement. Instead, we turn to its derivative for more practical results
\begin{corollary}\label{dPdz}
Let $G$ be a non-empty graph, and $\Pi:=\sum_{p\in \mathcal{P}_G}p$. 
Then 
$$
D\Pi=D\,\perm(\mathsf{I}+\mathsf{W})\ast \det(\mathsf{I}-\mathsf{W})=-\perm(\mathsf{I}+\mathsf{W})\ast D\det(\mathsf{I}-\mathsf{W}).
%\sum_{H\prec G}\det(\mathsf{I}-\mathsf{W}_{H})D \{ \perm(\mathsf{W}_{G-F})=-\sum_{H\prec G}\perm(\mathsf{W}_{H})D \{ \det(\mathsf{I}-\mathsf{W}_{G-F}).
$$
\end{corollary}
\begin{proof}
Derivating the expression for $\Pi$ given in Theorem~\ref{thm:main} yields
%\begin{subequations}
\begin{align*}%\label{stepdetperm1}
D \Pi&=D\,\perm(\mathsf{I}+\mathsf{W})\ast \perm(\mathsf{I}+\mathsf{W})^{\ast-1},\\%\label{stepdetperm1a}\\
&=-D\det(\mathsf{I}-\mathsf{W})\ast \det(\mathsf{I}-\mathsf{W})^{\ast-1}.%\label{stepdetperm1b}
\end{align*}
%\end{subequations}
The results follow on noting that $\ast$ is commutative and $\perm(\mathsf{I}+\mathsf{W})^{\ast-1}=\det(\mathsf{I}-\mathsf{W})$, as shown in Corollary~\ref{cor:1}, in the proof of Theorem~\ref{thm:main} and in Remark~\ref{GenRel}.
%In Eq.~(\ref{stepdetperm1a}), $\perm(\mathsf{I}+\mathsf{W})^{\ast-1}$ designates the $\ast$-inverse of $\perm(\mathsf{I}+\mathsf{W})$ and similarly in Eq.~(\ref{stepdetperm1b}). 
%Since the permanent is the identity on $\mathcal S$, its $\ast$-inverse must thus be the antipode on $\mathcal S$, i.e. $\perm(\mathsf{I}+\mathsf{W})^{-\ast1}=\det(\mathsf{I}-\mathsf{W})$.%
\end{proof}

%P(z)&=\sum_k \frac{(-1)^k}{k}(\perm(\mathsf{I}+\mathsf{W})-1)\ast(\perm(\mathsf{I}+\mathsf{W})-1\big)^{\ast k-1},\\
%&=\sum_k \frac{(-1)^k}{k}\sum_{H\prec G}\perm(z W_F)(\perm(\mathsf{I}+\mathsf{W})-1\big)^{\ast k-1}|_{G-F},\\
%&=\sum_{H\prec G}\perm(z W_F)\sum_k \frac{(-1)^k}{k}(\perm(\mathsf{I}+\mathsf{W})-1\big)^{\ast k-1}|_{G-F}
%
\begin{remark}[Practical considerations]
In practice, prime counting is achieved upon replacing all labeled adjacency matrices with ordinary adjacency matrices $\mathsf{W}\mapsto z\mathsf{A}$, with $z$ a formal variable. In this situation, formal polynomials on hikes become ordinary generating functions and the derivative operator $D$ is implemented as a derivative with respect to $z$. Then, because of the commutativity of the induced subgraph convolution, Corollary~\ref{dPdz} yields the following variant formulas for the derivative of the ordinary generating function of the primes $\Pi(z):=\sum_{p\in\mathcal{P}_G} z^{\ell(p)}$, 
\begin{align*}
\frac{d\,\Pi(z)}{dz}&=\sum_{H\prec G}\frac{d}{dz}\perm(\mathsf{I}+z\mathsf{A}_H)\det(-z\mathsf{A}_{G-H})=\sum_{H\prec G}\frac{d}{dz}\perm(z\mathsf{A}_H)\det(\mathsf{I}-z\mathsf{A}_{G-H}),\\
\frac{d\,\Pi(z)}{dz}&=-\sum_{H\prec G}\perm(\mathsf{I}+z\mathsf{A}_H)\frac{d}{dz}\det(-z\mathsf{A}_{G-H})=-\sum_{H\prec G}\perm(z\mathsf{A}_H)\frac{d}{dz}\det(\mathsf{I}-z\mathsf{A}_{G-H}).\end{align*}
%Replacing all labeled adjacency matrices with ordinary adjacency matrices, the above formulas give the number $\pi(\ell)$ of simple cycles of any length $\ell$ on $G$. More precisely, in this case, the coefficient of $z^\ell$ in the derivative of $P(z)$ is $\ell\, \pi(\ell)$.
\end{remark}
%~
%\pl{Is this useful? May be not or perhaps as a footnote}
%In the reduced incidence algebra of $\mathcal S$, the zeta function is $\zeta_{\ast}=\sum_{h\in \mathcal S}h\,z^{\ell(h)}=\perm(\mathsf{I}+\mathsf{W})$ and since $\mathcal S$ is commutative, it admits a Euler $\ast$-product form
%$$
%\zeta_{\mathcal S}(z)=\perm(\mathsf{I}+\mathsf{W})=\underset{\text{\tiny $h$ prime}}{\scalebox{2.2}{\raisebox{-.2ex}{$\ast$}}_{\hspace{-.5mm}G}} \big(1+z^{\ell(h)} h\big).
%$$
%\pl{Well the log formula underlying PV is also the BCH formula under disguise, i.e. BCH implements the Eulerian idempotent just as the log do}
~

\section{Simple cycles from Lie idempotents}\label{Lie}

The celebrated Milnor-Moore theorem \cite{Milnor1965} provides an explicit relation between connected graded cocommutative Hopf algebras and Lie algebras. In this section we exploit this relation to assert the existence of many more formulas for counting simple cycles on graphs. We illustrate this with two examples.\\ 
%
%
%
%
%Hopf algebraic constructions of the preceding section suggest that a further generalisation is possible to generate connected hikes from arbitrary ones and primes from self-avoiding ones. In this section we show that using Lie algebras, the $\ast$-logarithm of Theorem~\ref{thm:main} appears as just one among many objects, the Lie idempotents, all of which project $K\langle \H\rangle$ onto  $K\langle \mathcal{C}\rangle$ and $K\langle \mathcal S\rangle$ onto $K\langle \mathcal{P}\rangle$. It follows that each Lie idempotent provides a formula for counting the primes on any graph.

The commutative and cocommutative Hopf algebra $\S$ introduced in the previous section is graded, with gradation $c(h)=\Omega(h)$, and connected since $c(h)=0\iff h=1$ so that $K\langle\S_G|_{c(h)=0}\rangle$ is just $K$ itself as required \cite{Montgomery1993}.
Hence, we can use the theorem of Milnor and Moore to obtain that $\S_G$ is isomorphic to the universal enveloping algebra of the free graded Lie algebra formed by the primitive elements of $\S_G$, that is the simple cycles,
$$
\S_G \simeq U(\mathcal{P}_G).
$$
Taking $K$ to be a field with characteristic zero, these observations extend by linearity to  $K\langle\S\rangle$ and we
%which is indeed both graded and connected since $c(h)=0\iff h=1$ so that $K\langle\S|_{c(h)=0}\rangle$ is just $K$ itself as required \cite{Montgomery1993}. 
 thus have $K\langle\S_G\rangle \simeq U(K\langle\mathcal{P}_G\rangle)$.\\

These results provide new tools to pass from $\S_G$ to the free Lie algebra formed by its primitive elements: the Lie idempotents.
%By linearity these idempotents will also pass from $K\langle \S\rangle$ to $K\langle \mathcal P\rangle$ giving rise to at the level of $K$-linear combinations\\ 
Lie idempotents are symmetrizers projecting the tensor algebra $T(A)$ of a Lie algebra $A$ onto the free Lie algebra. Now recall that the universal enveloping algebra of the Lie algebra $A$ is $U(A)=T(A)/I$, with $I$ the two-sided ideal generated by elements of the form $a\otimes b-b\otimes a-[a,b]$. In particular, if $A$ is free, then a Lie idempotent projects $U(A)$ onto the free Lie-algebra on the $K$-module $A$. Since $\mathcal{P}_G$ is free, this reasoning leads to: 
\begin{theorem}\label{LieTheorem}
Let $\imath$ be a Lie idempotent. Then 
\begin{align*}
\imath: &~\S_G\,\longrightarrow \mathcal{P}_G,
\end{align*}
and by linearity $\imath:\,K\langle \mathcal{S}_G\rangle \,\longrightarrow K\langle \mathcal{P}_G\rangle$.
\end{theorem}

\noindent We now give two examples of Lie idempotents to illustrate this result.
 
\begin{example}[Eulerian idempotent]
Let $A$ be a cocommutative connected graded $K$-bialgebra with product $\star$ and let $\text{id}_A$ be the identity map on $A$. Then the endomorphism $\mathfrak{e}:=\log_\star(\text{id}_A)$ projects $A$ onto the $K$-submodule of primitive elements and is called the Eulerian idempotent of $A$ \cite{Grinberg2014}. Theorem~\ref{thm:main} states the Eulerian idempotents on $\mathcal{S}_G$ and $K\langle \S_G\rangle$.\\
\end{example}

\begin{example}[Dynkin idempotent]
Let $K$ a commutative $\mathbb{Q}$-algebra  and $A$ be a cocommutative connected graded $K$-Hopf algebra with product $\star$. Let $S$ be antipode of $A$ and for any $a\in A$ define $E(a):=deg(a)a$, with $deg(a)$ the grade of $a$. Then the endomorphism of $A$ denoted $\mathfrak{d}:=
S\star E$ projects $A$ onto the $K$-submodule of primitive elements and is called the Dynkin idempotent of $A$ \cite{Waldenfels1966,Patras2002,Grinberg2014}. In the context of the self-avoiding hikes, $deg(h)=c(h)=\Omega(h)=\omega(h)$, with $\omega$ the number of distinct prime factors of $h$. 
Thus the Dynkin idempotent on $\mathcal{S}_G$ reads
%\footnote{A similar, though less interesting, relation holds on $\H$ with both $\Omega(h)$ and $\omega(h)$ replaced by $c(h)$. It yields the identity on $\mathcal{C}$.}
$$
\Pi = \sum_{h\in \mathcal{S}_G} (-1)^{\Omega(h)} h\, \sum_{h\in \mathcal{S}_G} \omega(h) h.
$$
In fact, this result is recovered from a straightforward argument in the reduced incidence algebra of $\mathcal{S}_G$. Indeed, a direct multiplication of $\Pi$ with the zeta function of $\mathcal{S}_G$ gives
$$
\Pi\,\zeta_{\mathcal{S}_G}= \sum_{h\in \mathcal{S}_G} \bigg(\,\sum_{p\in \mathcal{P}_G,~p|h}1\bigg) h =\sum_{h\in \mathcal{S}_G}\omega(h) h,
$$ 
and the Dynkin idempotent follows after a M\"{o}bius inversion of the above relation.
\end{example}

Many more Lie idempotents have been discovered and can be found in the relevant literature, see e.g. \cite{Patras1999, Patras2002, Thibon2016, Grinberg2014} and references therein. By Theorem~\ref{LieTheorem}, each one of them provides a formula for counting the primes, that is the simple cycles, on arbitrary weighted directed graphs. 
%
%\section{Conclusion}

\section*{Acknowledgements}
P.-L. Giscard is grateful for the financial support provided by the Royal Commission for the Exhibition of 1851.

\section*{References}
\bibliographystyle{plain}

\end{document}